\documentclass[12pt,sumlimits,intlimits]{amsart}
\usepackage{url,cite,hyperref,amsmath,amsthm,amssymb,mathtools}
\usepackage[margin=1.25in]{geometry}
\usepackage{xcolor}
\usepackage{comment}

\newtheorem{thmx}{Theorem}

\newtheorem{theorem}{Theorem}[section]
\newtheorem*{theorem*}{Theorem}
\newtheorem{proposition}[theorem]{Proposition}
\newtheorem{corollary}[theorem]{Corollary}
\newtheorem{lemma}[theorem]{Lemma}
\theoremstyle{definition}
\newtheorem{definition}[theorem]{Definition}
\newtheorem{remark}[theorem]{Remark}
\newtheorem{example}[theorem]{Example}

\newcommand{\la}{\langle}
\newcommand{\ra}{\rangle}

\newcommand{\Z}{\mathbb{Z}}

\newcommand{\C}{\mathbb{C}}

\newcommand{\R}{\mathbb{R}}

\newcommand{\T}{\mathbb{T}}

\newcommand{\ve}{\varepsilon}

\newcommand{\id}{\operatorname{id}}

\numberwithin{equation}{section}

\title{C*-algebras generated by representations of virtually nilpotent groups}
\author{Caleb Eckhardt}
\address{Department of Mathematics, Miami University, Oxford, OH, 45056}
\email{eckharc@miamioh.edu}

\begin{document}
\maketitle
\begin{abstract} We show that a C*-algebra generated by an irreducible representation of a finitely generated virtually nilpotent group satisfies the universal coefficient theorem and has real rank 0. This combines with previous joint work with Gillaspy and McKenney to show these C*-algebras are classified by their Elliott invariants.  
 When we further assume the group is nilpotent we build explicit Cartan subalgebras that are closely related to the group and representation, although the Cartan subalgebras are generally not C*-diagonals.
\end{abstract}
\section{Introduction}  This paper is primarily motivated by the Elliott classification program for nuclear C*-algebras. A C*-algebra is \textbf{classifiable} if it is separable, nuclear, simple, has finite nuclear dimension and satisfies the universal coefficient theorem (UCT). This terminology is derived from the fact that such algebras are \emph{classified} by their Elliott invariants \cite{Gong20,White23}.  

The main theorems of \cite{EckhardtGillaspy2016Irreducible}  and  \cite{EckhardtMcKenney2018Finitely} show that all irreducible representations of finitely generated nilpotent groups generate classifiable C*-algebras.  The main purpose of this work is to extend this to the \emph{virtually} nilpotent case. We show
\begin{thmx} \label{thm:A}  Let $G$ be a finitely generated virtually nilpotent group and $\pi$ an irreducible unitary representation of $G.$  Then the C*-algebra generated by $\pi(G)$ is classifiable, has a unique trace and real rank 0. 
\end{thmx}
 Theorem \ref{thm:A} is  definitive in the following sense.  A finitely generated group $G$ is virtually nilpotent if and only if every irreducible representation of $G$ generates a simple nuclear C*-algebra (see \cite{Echterhoff90}).  In other words virtually nilpotent groups are the largest possible class of finitely generated groups to which one could hope to apply the Elliott classification program to all the C*-algebras generated by its irreducible representations.  

Theorem \ref{thm:A} is the culmination of the papers \cite{Eckhardt14,EckhardtGillaspy2016Irreducible,EckhardtMcKenney2018Finitely,EckhardtGillaspyMcKenney2019Finite}. Allow me to discuss some old roadblocks and new observations necessary to extend classifiability from the nilpotent case to the virtually nilpotent case.
Our joint work with Gillapsy and McKenney \cite{EckhardtGillaspyMcKenney2019Finite} shows finite nuclear dimension and therefore reduces our task to showing that such C*-algebras satisfy the UCT.

Let $G$ be a finitely generated nilpotent group, $\pi$ an irreducible representation of $G$ and $C^*_\pi(G)$ the C*-algebra generated by $\pi(G).$  In \cite{EckhardtGillaspy2016Irreducible} Elizabeth Gillaspy and I showed that  $C_\pi^*(G)$ decomposes as a crossed product of a Type I C*-algebra by iterated integer actions. The UCT then followed from Rosenberg and Schochet's \cite{RosenbergSchochet1987UCT}.  Using the well-known correspondence between traces and irreducible representations of $G$ (see e.g. \cite[Section 2]{EckhardtGillaspyMcKenney2019Finite})  one notices that \cite{EckhardtGillaspy2016Irreducible} really\footnote{Also important was $G/N$ being a poly-$\Z$ group although poly-$\Z$ can now be replaced with amenability by Theorem \ref{thm:UCTindtraces}. On the other hand,  for the construction of Cartan subalgebras in Theorem \ref{thm:B}, we rely heavily on the fact that $G/N$ is torsion free.} boils down to the fact that if $\tau$ is an extreme faithful trace on $G$, then there is a normal, virtually abelian subgroup $N\trianglelefteq G$ such that $\tau(G\setminus N) = \{ 0\}$ (see \cite[Theorem 4.5]{Carey84}). This is false for \emph{virtually} nilpotent groups  and I thank Itamar Vigdorovich for showing me this with the following 

\begin{example} \label{ex:nonfaithful} Let $G$ be a group and $\tau$ a faithful character (Section \ref{sec:trandcp}) on $G.$  Let $H=G^2 \rtimes \Z/2\Z$ where the action flips the $G$ factors.  Define $\phi:G^2\rightarrow \C$ by
$\phi(x,y) = \frac{1}{2}(\tau(x)+\tau(y)).$  Let $\phi$ also denote the trivial extension to $H.$ Then $\phi$ is clearly faithful.
It is straightforward that $\pi_\phi(H)'' \cong M_2\otimes \pi_\tau(G)''$ and hence $\phi$ is a  character. If $G$ is nilpotent then $H$ is virtually nilpotent and $\phi(x,e)\neq 0$ for all $x\in G.$ Hence if $G$ is 
a torsion free non-abelian nilpotent group (for example the integer Heisenberg group) and $\tau$ is induced from a faithful multiplicative character on the center, then there is no normal virtually abelian subgroup $N\trianglelefteq H$ such that $\phi(H\setminus N) = \{ 0 \}.$
\end{example}

In hindsight traces like this are what stymied our efforts in \cite{EckhardtGillaspy2016Irreducible} to extend to the virtual case. After digesting Example \ref{ex:nonfaithful}  I realized that  if $G$ is virtually nilpotent, then  $C^*_\pi(G)$ decomposes as a twisted crossed product of a finite dimensional C*-algebra by an amenable group (modulo some minor modifications). The bulk of this paper is devoted to detailing that decomposition.  The UCT then more-or-less reduces to an application of Tu's theorem on amenable groupoids \cite{Tu99}.  While I do not directly apply any of the results
in \cite{Levit22}, I was certainly inspired by the induced trace ideas used by Levit and Vigdorovich in the proof of \cite[Theorem G]{Levit22}.

Finally, an interesting byproduct of this work is the construction of explicit Cartan subalgebras 

\begin{thmx}\label{thm:B} Let $G$ be a finitely generated  nilpotent group and $\pi$ an irreducible representation of $G.$  Then there is a torsion free nilpotent group $H$,  a 2-cocycle $\sigma\in Z^2(H,\T)$ and an $n\geq 1$ such that
\begin{equation*}
C^*_\pi(G) \cong M_n\otimes C_r^*(H,\sigma)
\end{equation*}
Moreover there is a subgroup $N\leq H$ such that $D_n\otimes C_r^*(N,\sigma)$ is a Cartan subalgebra of $M_n\otimes C^*(H,\sigma)$ where $D_n$ is a diagonal MASA in $M_n.$
\end{thmx}
  Xin Li showed \cite{Li20} that every classifiable C*-algebra has a Cartan subalgebra--in fact he showed in the stably finite case that they have C*-diagonals. We already knew $C^*_\pi(G)$ was classifiable by \cite{EckhardtGillaspy2016Irreducible} so  the existence of a C*-diagonal was guaranteed by Li's theorem.   The point then of the last bit in Theorem \ref{thm:B} is to explicitly construct a Cartan subalgebra that is closely related to the group structure and representation. We do pay a price though;  the Cartan subalgebras we construct are typically not C*-diagonals (see Example \ref{ref:nondiagonalsimple}).

\section{Results} 
\subsection{Notation} We refer the reader to \cite{Baumslag71Nilpotent} for the basic definitions and results about nilpotent groups and to \cite[Section 2]{EckhardtGillaspyMcKenney2019Finite} for basic information on the C*-algebras generated by their representations.  Colloquially a group is nilpotent if it is ``built up" by central extensions. A group is virtually nilpotent if it contains a finite index nilpotent subgroup.
 
When studying representations of groups built up by extensions, the machinery of twisted crossed products is essential to understanding the structure of the C*-algebras they generate. We briefly discuss this machinery--just enough to fix our notation--and refer the reader to the standard reference
\cite{PackerRaeburn1989twisted} for more information.  Since we only deal with discrete groups and the canonical conditional expectation is a key point the reader may prefer to consult
Section 2 of \cite{Bedos91} instead.

Let $A$ be a unital C*-algebra and $G$ an amenable discrete group.  Let $(G,\alpha,\omega)$ be a twisted action where $\alpha$ denotes the automorphisms and $\omega$ denotes the cocycle. We denote by $A\rtimes_{\alpha,\omega} G$ the \textbf{reduced twisted crossed product} (see \cite[Section 2]{Bedos91}).\footnote{Some authors use $\alpha,\omega,r$ for the reduced crossed product and $\alpha,\omega$ for the full crossed product.  All of our crossed products will be reduced so we drop the ``$r$"}  We let $\{u_g:g\in G\}\subseteq A\rtimes_{\alpha,\omega} G$ denote the unitaries implementing $\alpha,$ i.e. $\alpha_g = \text{Ad } u_g.$  There is a faithful conditional expectation $E:A\rtimes_{\alpha,\omega} G\rightarrow A$ defined by 
$E(xu_g)=0$ if  $g\neq e$ and $x\in A$ that we call the \textbf{canonical conditional expectation} (see e.g. \cite[Theorem 2.2]{Bedos91}).  If $\tau$ is an $\alpha$-invariant  trace on $A$ then the map $\tau\circ E$ is a trace on $A\rtimes_{\alpha,\omega} G,$ which is faithful whenever $\tau$ is.  We will use the same letter $\tau$ to stand for the trace on $A$ and on the crossed product and call it the \textbf{canonical tracial extension of }$\tau.$ Notice that if $\phi$ is a faithful state then $\phi\circ E$ is a faithful state on $A\rtimes_{\alpha,\omega} G.$

In the special case that $A=\C$, the twisted crossed product is simply a twisted group C*-algebra \cite{PackerRaeburn1992structure} and we switch to the standard notation $C^*_r(G,\sigma)$ where $\sigma\in Z^2(G,\T)$ is a 2-cocycle. In this case we denote the canonical generating unitaries as $\{ \lambda_g:g\in G  \}$ since we  sometimes deal simultaneously with twisted group C*-algebras and twisted crossed products of the same group.

\subsection{Tracial representations and twisted crossed products} \label{sec:trandcp} A \textbf{trace} on a group $G$ is a normalized positive definite function $\tau:G\rightarrow \C$ that is constant on conjugacy classes. There is a well-known correspondence between traces on $G$ and traces on $C^*(G)$ given by linearization and restriction.  Following representation theory convention (instead of C*-convention) we call extreme traces \textbf{characters}.  In general we use the same letter $\tau$ to stand for a trace on a group or on $C^*(G).$ Notice that $\ker(\tau) = \{ g\in G:\tau(g)=1 \}$ is a normal subgroup of $G.$  We say $\tau$ is \textbf{faithful} if $\ker(\tau)=\{ e \}.$
Notice that $\tau$ also defines a trace on the quotient $G/\ker(\tau).$  In general we use the same letter $\tau$ to denote the trace on $G$ and on the quotient $G/\ker(\tau).$

Let $C^*_\tau(G)$ be the C*-algebra generated by the GNS representation associated with $\tau.$  There is a clear connection between $C^*_\tau(G)$ and twisted crossed products that we recall in the following
\begin{remark}\label{rem:tracetotwist} Let $G$ be a discrete amenable group and $N\trianglelefteq G$ a normal subgroup.  Let $\tau$ be a trace on $G$ such that $\tau(g)=0$ for all $g\not\in N.$  Let $c:G/N\rightarrow G$ be a choice of coset representatives sending the identity of $G/N$ to the identity of $G.$  For each $t\in G/N$ let $\alpha_t = \text{Ad}(\pi_\tau(c_t))\in \text{Aut}(C^*_\tau(N)).$  Define $\omega(s,t) = c_sc_tc_{st}^{-1}.$  Then $(\alpha,\omega)$ defines a twisted action of $G/N$ on $C^*_\tau(N)$ and the maps
\begin{equation*}
 \id:C^*_\tau(N)\to C^*_\tau(N), \quad u_t\mapsto \pi_\tau(c_t)
 \end{equation*}
 form a covariant representation. The condition $\tau(g)=0$ when $g\not\in N$ shows that this map preserves the  canonical tracial extension of $\tau$  and hence gives an isomorphism 
 \begin{equation*}
C^*_\tau(N)\rtimes_{\alpha,\omega}G/N\cong C^*_\tau(G).
 \end{equation*}
 Finally we mention a useful fact that is a tedious application of the definitions.  It also showcases a flexibility of twisted crossed products that is absent in regular crossed products. Let $A$ be any C*-algebra  with a twisted action $(\alpha,\omega)$ of $G.$  Let $N\trianglelefteq G$ be a normal subgroup.  Then there is an (obvious) twisted action $(\alpha',\omega')$ of $G/N$ on $A\rtimes_{\alpha|_N,\omega|_N} N$ such that
 \begin{equation*}
 A\rtimes_{\alpha,\omega} G \cong (A\rtimes_{\alpha|_N,\omega|_N} N)\rtimes_{\alpha'\omega'} G/N
 \end{equation*}
\end{remark}
\subsection{Further structure of twisted crossed products}
The following  is a special case of  \cite[Theorem 2.13]{Green1980imprimitivity}. In this case the proof is brief so we include it--in the modern language of twisted crossed products--for the convenience of the reader.

\begin{proposition}\label{prop:transitive} Let $A=A_1\oplus \cdots \oplus A_n$ be a unital C*-algebra such that each $A_i$ has no non-trivial central projections.  Let $(G,\alpha,\omega)$ be a twisted action on $A$ with $G$ a discrete group.  Let $\{ p_1,...,p_n \}$ be the minimal central projections of $A$ corresponding to the summands $A_1,...,A_n.$  Then $\alpha$ defines a true action of $G$ on  $\{ p_1,...,p_n \}.$  Let $H\leq G$ be the stabilizer of $p_1.$ If  $G$ acts transitively on $\{ p_1,...,p_n \}$,  then
\begin{equation*}
A\rtimes_{\alpha,\omega} G \cong M_n\otimes (A_1 \rtimes_{\tilde{\alpha},\tilde\omega} H)
\end{equation*}
where $\tilde\alpha, \tilde\omega$ is the twisted action restricted to $H$ and $A_1,$ i.e. $\tilde \alpha_t(\cdot) : = \alpha_t(p_1 \cdot)$ and $\tilde \omega(s,t) = p_1\omega(s,t) $ for all $s,t\in H.$
\end{proposition}
\begin{proof} Notice that by definition of $\tilde \alpha,\tilde\omega$ we have 
\begin{equation}\label{eq:tcpassub}
A_1 \rtimes_{\tilde{\alpha},\tilde\omega} H \cong p_1(A\rtimes_{\alpha,\omega} G)p_1
\end{equation}
Fix a faithful state $\phi$ on $A\rtimes_{\alpha,\omega} G$ such that $\phi = \phi\circ E$ where $E$ is the canonical conditional expectation. 
Choose $e = t_1,...,t_n\in G$ such that
\begin{equation}\label{eq:unitarychoice}
u_{t_i}p_1u_{t_i}^* = p_i.
\end{equation}
Easy calculations show that 
\begin{equation*}
e_{i,j} = u_{t_i}p_1u_{t_j}^* \text{ for }1\leq i,j\leq n
\end{equation*}
form a system of matrix units and hence generate a (unital) copy of $M_n$ inside of $A\rtimes_{\alpha,\omega} G.$  Under the identification (\ref{eq:tcpassub}), the map $\pi: A_1 \rtimes_{\tilde{\alpha},\tilde\omega} H\rightarrow A\rtimes_{\alpha,\omega} G$ defined by
\begin{equation*}
\pi(a) = \sum_{i=1}^n   u_{t_i}au_{t_i}^*
\end{equation*}
 is an injective *-homomorphism.  Let $a\in p_1(A\rtimes_{\alpha,\omega} G)p_1.$ For each $j=1,...,n$ we have (using the facts that $t_1=e$ and $ap_1=p_1a=a$) that
\begin{equation*}
u_{t_j}p_1\left(\sum_{i=1}^n u_{t_i}au_{t_i}^*\right)=u_{t_j}p_1a = u_{t_j}au_{t_j}^*u_{t_j}p_1 = \left(\sum_{i=1}^n u_{t_i}au_{t_i}^*\right)u_{t_j}p_1.
\end{equation*}
Since the ``column" $\{ u_{t_j}p_1: j=1,...,n  \}$ generates $M_n$ the range of $\pi$ and $M_n$ commute with each other. To complete the proof we show that the matrix units and the image of $\pi$ generate $A\rtimes_{\alpha,\omega} G.$ 

It is clear that $A\subseteq C^*(\text{Im}(\pi),e_{11},...,e_{nn}).$  Let $x\in G$ and fix $1\leq i\leq n$ and let $j$ satisfy $u_xp_iu_x^* = p_j.$   Notice that $t_j^{-1}xt_i \in H.$  We have
\begin{align*}
u_xp_i & = p_ju_xp_i\\
& =p_j(u_{t_j}p_1u_{t_j}^*)u_x(u_{t_i}p_1u_{t_i}^*)p_i \quad \text{ by (\ref{eq:unitarychoice}) }\\
& = (p_ju_{t_j}p_1)u_{t_j}^*u_xu_{t_i}(p_1u_{t_i}^*p_i)\\
& = (p_ju_{t_j}p_1)au_{t_j^{-1}xt_i}(p_1u_{t_i}^*p_i)\quad \text{  where }a\in A \text{ comes from the cocycle}\\
& = (u_{t_j}p_1)p_1\pi(p_1au_{t_j^{-1}xt_i})(p_1u_{t_i}^*)
\end{align*}
Hence $u_xp_i$ is a product of matrix units and elements in the range of $\pi.$  Since $u_x = \sum_{i=1}^n u_xp_i$ this completes the proof.
\end{proof}
A version of the following is known to anyone who has played around with an inner action on a *-algebra. Here is a brief proof for those readers who have not.
\begin{proposition}\label{prop:inner} Let $A$ be a unital C*-algebra with trivial center and $G$ a discrete amenable group with a twisted point-wise inner action $(\alpha,\omega)$ on $A.$  Then  there is a 2-cocycle $\sigma\in Z^2(G,\T)$ such that $A\rtimes_{\alpha,\omega} G\cong A\otimes_{\textup{min}} C^*_r(G,\sigma) .$
\end{proposition}
\begin{proof} For each $t\in G\setminus\{ e \}$ choose a unitary $v_t\in A$ such that $\text{Ad}(v_t) = \text{Ad}(u_t)$ and set $v_e = 1_A.$  Then for each $s\in G$ we have $v_s^*u_s\in A'.$ Then for $s,t\in G$ we have
\begin{equation*}
\sigma(s,t):= (v_s^*u_s)(v_t^*u_t)(u_{st}^*v_{st}) = v_s^*\alpha_s(v_t^*)\omega(s,t)v_{st} \in A\cap A' =\C\cdot 1_A.
\end{equation*}
It follows that $\sigma\in Z^2(G,\T).$  
Let $\phi$ be a faithful state on $A\rtimes_{\alpha,\omega} G$ such that $\phi\circ E = \phi$ where $E$ is the canonical conditional expectation onto $A.$ Let $\tau$ be the canonical faithful trace on $C_r^*(G,\sigma).$ Then $\phi|_A\otimes \tau$ is a faithful state on $A\otimes_{\text{min}}C_r^*(G,\sigma).$  
The map $\pi(t) =  v_t^*u_t$ is a $\sigma$-representation of $G$ and $0=\tau(\lambda_t) = \phi(v_t^*u_t)$, for $t\neq e$ hence $\pi$ extends to an injective *-homomorphism of $C^*_r(G,\sigma).$  Then $\pi$ and $\iota:A\to A\subseteq A\rtimes_{\alpha,\omega} G$ have commuting ranges.  Since $C^*_r(G,\sigma)$ is nuclear we obtain a surjective *-homomorphism  $\iota\otimes \pi: A\otimes_{\text{min}}C^*_r(G,\sigma)\to A\rtimes_{\alpha,\omega} G.$ Since $\phi\circ (\iota\otimes \pi) = \phi|_A\otimes \tau$, the map is injective as well and therefore a *-isomorphism. 
\end{proof}

\subsection{UCT and classifiability}
Let $\tau$ be a trace on an amenable group $G$.  Roughly speaking, we find that the larger the zero set of $\tau$ the easier it is to show $C^*_\tau(G)$ satisfies the UCT.  The next four results smooth out that statement a little bit.

\begin{theorem}\label{thm:vnilcase} Let $G$ be a discrete amenable group and $\tau$ a trace on $G.$ Suppose there is a normal subgroup $N\trianglelefteq G$ such that
\begin{enumerate}
\item $C^*_\tau(N)$ is finite dimensional, and 
\item $\tau(g)=0$ for all $g\not\in N$
\end{enumerate}
Then there are finite index subgroups $H_1,...,H_n \leq G/N,$ cocycles $\sigma_i\in Z^2(H_i,\T)$ and integers $k_1,...,k_n$ such that
\begin{equation*}
C^*_\tau(G) \cong \bigoplus_{i=1}^n M_{k_i}\otimes C^*_r(H_i,\sigma_i)
\end{equation*}
In particular, $C^*_\tau(G)$ satisfies the UCT.
\end{theorem}
\begin{proof} First realize $C^*_\tau(G)\cong C^*_\tau(N)\rtimes_{\alpha,\omega}G/N$ as in Remark \ref{rem:tracetotwist}. Since $C^*_\tau(N)$ is finite dimensional we decompose 
\begin{equation*}
C^*_\tau(N) \cong \bigoplus_{i=1}^n F_i
\end{equation*}
where each $F_i$ is a finite dimensional C*-algebra and 
\begin{enumerate}
\item Each $\alpha_t$ leaves the subalgebra $F_i$ globally invariant.
\item The (true) action of $G/N$ on the minimal central projections of $F_i$ is transitive.
\end{enumerate}
We then have that 
\begin{equation*}
C^*_\tau(N)\rtimes_{\alpha,\omega}G/N \cong \bigoplus_{i=1}^n (F_i\rtimes_{\alpha_i,\omega_i} G/N)
\end{equation*}
where $\alpha_i,\omega_i$ is the action restricted to $F_i.$  Since every automorphism of a matrix algebra is inner, the structure result follows from Propositions \ref{prop:transitive} and \ref{prop:inner}.  By the twisted version \cite[Theorem 3.1]{Barlak17} of Tu's \cite{Tu99}, the twisted group C*-algebras $C^*_r(H_i,\sigma_i)$ all satisfy the UCT.  Since the UCT is preserved under stable isomorphism and direct sums the conclusion follows.
\end{proof}
Weakening the hypotheses in Theorem \ref{thm:vnilcase} we forfeit the structure of twisted group C*-algebras but retain the UCT.  Theorem \ref{thm:UCTindtraces} will not be used to prove our main theorem but it should find future applications. We first require the following consequence of Tu's  \cite{Tu99}.  Barlak and Li give an exposition of some of Tu's argument in the more general setting of twisted groupoids and their exposition carries over nearly verbatim to show the following.  In our proof we only mention the slight changes needed from the proof in  \cite[Theorem 3.1]{Barlak17} and keep their notation.
This lemma and arguments used in the proof are apparently well-known to experts, I thank Rufus Willett for showing it to me. The same result appears as a corollary of a more general theorem in \cite[Corollary 6.2]{Kwasniewski86}.
\begin{lemma} \label{lem:Tucprod} Let $G$ be an amenable group with a twisted action on a Type I C*-algebra $B$.  Then the  twisted crossed product of $B$ with $G$ satisfies the UCT.
\end{lemma}
\begin{proof}  By the Packer-Raeburn trick \cite[Theorem 3.4]{PackerRaeburn1989twisted} we can and do assume that the twisted action is a true action. 

Let $\mathcal{A}(H)$ be the proper $G$-C*-algebra constructed by Tu \cite{Tu99} and used in \cite[Theorem 3.1]{Barlak17}.  
For our proof we only need to know that
 $\mathcal{A}(H)$ is an inductive limit of Type I proper $G$-C*-algebras $\{\mathcal{A}_\ve(H)\}_{\ve>0}.$
By replacing the compact operators (denoted by \textbf{A} in \cite[Theorem 3.1]{Barlak17})\footnote{Note that in our special case of a group action and not a groupoid that \textbf{A} really is the compact operators} with the Type I C*-algebra $B$ we see that $B\rtimes G$ is $KK$-dominated (\cite[Definition 23.10.6]{Blackadar98}) by $(\mathcal{A}(H)\otimes B)\rtimes G.$ Hence (as in \cite[Theorem 3.1]{Barlak17}) we only need to show that $(\mathcal{A}(H)\otimes B)\rtimes G$ satisfies the UCT.  But $(\mathcal{A}(H)\otimes B)\rtimes G$ is an inductive limit of $(\mathcal{A}_\ve(H)\otimes B) \rtimes G$ hence we only need to show that each $(\mathcal{A}_\ve(H)\otimes B) \rtimes G$ satisfies the UCT.  But since $\mathcal{A}_\ve(H)\otimes B$ is Type I and the action is proper the crossed product is also Type I by \cite[Theorem 2.14]{Echterhoff11} and hence satisfies the UCT by \cite{RosenbergSchochet1987UCT}.
\end{proof}
\begin{theorem} \label{thm:UCTindtraces}  Let $G$ be an amenable group and $\tau$ a trace on $G.$ Let $N\trianglelefteq G$ be a virtually abelian normal subgroup such that $\tau(g)=0$ when $g\not\in N.$  Then $C^*_\tau(G)$ satisfies the UCT.
\end{theorem}
\begin{proof} Since $N$ is virtually abelian, $C^*_\tau(N)$ is subhomogeneous and hence Type I.  The conclusion then follows by Lemma \ref{lem:Tucprod} and Remark \ref{rem:tracetotwist}.
\end{proof}

\begin{definition}\label{def:trivext} Let $G$ be a group and $N\trianglelefteq G$ be a normal subgroup.  Let $\tau$ be a $G$-invariant trace on $N.$  The \textbf{trivial extension of }$\tau$ is defined as 
\begin{equation*}
\tilde \tau(g) = \left\{ \begin{array}{ll}\tau(g) & \text{ if }g\in N\\
                                                           0 & \text{ if }g\notin N \end{array}\right.
\end{equation*}
Since $\tau$ is $G$-invariant it follows that $\tilde \tau$ is a trace on $G.$
\end{definition}
\begin{lemma}\label{lem:trivialextsuffices} Let $G$ be a virtually nilpotent group and $N\trianglelefteq G$ a finite index subgroup.  Let $\tau$ be a character on $G.$  Let $\phi$ be the trivial extension of $\tau|_N$ to $G.$  If $C^*_\phi(G)$ satisfies the UCT, then so does $C^*_\tau(G).$ 
\end{lemma}
\begin{proof} By  \cite[Proposition 4.4]{EckhardtGillaspyMcKenney2019Finite}, we have $C^*_\phi(G)\cong C^*_\tau(G)\oplus B$ for some C*-algebra $B$\footnote{A better description of $B$ is given in \cite[Proposition 4.4]{EckhardtGillaspyMcKenney2019Finite} but it doesn't concern us here}.  The conclusion now follows from \cite[Proposition 23.10.5]{Blackadar98}.
\end{proof}
\begin{remark} In general it is not the case that $\tau=\phi$ in Lemma \ref{lem:trivialextsuffices}.  It is preferable to work with traces that vanish on large parts of groups because we can employ the machinery of twisted crossed products.  With regards to the UCT Lemma \ref{lem:trivialextsuffices} provides that flexibility.
\end{remark}
\begin{remark} \label{rem:FCcenter} Let $G$ be a group.  The \textbf{FC-center} of $G$,  denoted $\text{FC}(G)$ is the collection of all elements of $G$ with finite conjugacy classes. 
Let now $N$ be finitely generated and nilpotent. Then $\text{FC}(N)$ contains $Z(N)$ as a finite index subgroup (see e.g. \cite[Lemma 2.3]{EckhardtGillaspy2016Irreducible}).  Hence if  $\tau$ is a character on $N$ then $\pi_\tau(\text{FC}(N))$ is finite dimensional. By \cite[Lemma 2.4]{EckhardtGillaspy2016Irreducible}, the quotient $N/\text{FC}(N)$ is torsion free.
\end{remark}

\begin{corollary}\label{cor:RR0nilpotent} Let $N$ be a finitely generated nilpotent group and $\pi$ an irreducible representation of $N$.  Then there is a finitely generated torsion free group  $H$,  a $\sigma\in Z^2(H,\T)$ and an $n\geq1$ such that
\begin{equation*}
C^*_\pi(N) \cong M_n\otimes C_r^*(H,\sigma).
\end{equation*}
If $H$ is nontrivial, then $C^*_\pi(H,\sigma)$ contains an irrational rotation algebra and has real rank 0.  If $H$ is trivial, then $C^*_\pi(N)$ (trivially) has real rank 0.
\end{corollary}
\begin{proof} Let $\tau$ be an extreme trace on $N$ such that $C^*_\tau(N)\cong C^*_\pi(N).$  By replacing $N$ with $N/\ker (\tau)$ we may assume that $\tau$ is faithful on $N.$ By \cite[Theorem 4.5]{Carey84} we have $\tau(g)=0$ for all $g\not\in \text{FC}(N).$  By Remark \ref{rem:FCcenter} we have that $C^*_\tau(\text{FC}(N))$ is finite dimensional and $N/\text{FC}(N)$ is torsion free.  Since $C^*_\pi(N)$ is simple, by Theorem \ref{thm:vnilcase} there is a subgroup $H\leq N/\text{FC}(N)$, a $\sigma\in Z^2(H,\T)$ and an $n\geq 1$ so $C^*_\pi(N)\cong M_n\otimes C^*_r(H,\sigma).$

Suppose now that $H$ is non-trivial.  By \cite{EckhardtMcKenney2018Finitely}, $C^*_r(H,\sigma)$ has a unique trace and is $\mathcal{Z}$-stable.  Therefore by \cite[Corollary 7.3]{Rordam04} it suffices to show that the range of the trace $K_0(\tau)$ is dense in $\R.$  Since the range of the trace on an irrational rotation algebra is dense in $\R$ (see \cite{Rieffel81}) it suffices to show that $C^*_r(H,\sigma)$ contains an irrational rotation algebra.

Fix $g\in Z(H)\setminus\{ e \}.$   Let $\tilde H = \la \lambda_h:h\in H  \ra.$  Then $\tilde H$ is finitely generated and a central extension of $H$ and hence nilpotent.   Therefore the torsion subgroup of $\tilde H$ is finite\footnote{We remark that the torsion subgroup is contained in $\T$}. 
Let $m$ be the order of the torsion subgroup of $\tilde{H}$

For any $h\in H,$ the group commutator $[\lambda_g,\lambda_h]=\lambda_g^{-1}\lambda_h^{-1}\lambda_g\lambda_h\in \T$ because $g$ is central in $H.$ Suppose that for all $h\in H$, the commutator $[\lambda_g,\lambda_h]$ has finite order.  Then $[\lambda_g^m,\lambda_h] = [\lambda_g,\lambda_h]^m=1$ for all $h\in H.$  Since $H$ is torsion free this shows that $\lambda_{g}^m$ is a non-trivial element of the center of $C^*_r(H,\sigma)$--a contradiction. Hence there is some $x\in H$ so $[\lambda_g,\lambda_x]$ has infinite order and therefore $C^*(\lambda_g,\lambda_x)$ is isomorphic to an irrational rotation algebra.
\end{proof}
\begin{corollary}\label{cor:RR0vniltwist} Let $G$ be a finitely generated virtually nilpotent group and $\sigma\in Z^2(G,\T)$ such that $C^*_r(G,\sigma)$ is simple.  Then $C^*_r(G,\sigma)$ has real rank 0.
\end{corollary}
\begin{proof} We prove the non-trivial case, where  $C^*_r(G,\sigma)$ is infinite dimensional. Let $N\trianglelefteq G$  be a finite index nilpotent subgroup. By \cite[Lemmas 3.5, 3.6]{EckhardtGillaspyMcKenney2019Finite}, $C_r^*(N,\sigma)$ is a direct sum of simple quotients of $C^*(N).$ It follows from Corollary \ref{cor:RR0nilpotent} that $C^*_r(G,\sigma)$ contains a (possibly non-unital) copy of an irrational rotation algebra.  Since $C^*_r(G,\sigma)$ has a unique trace and is $\mathcal{Z}$-stable (\cite{EckhardtGillaspyMcKenney2019Finite}) real rank 0 follows by \cite[Corollary 7.3]{Rordam04} as in the proof of Corollary \ref{cor:RR0nilpotent}.
\end{proof}
\begin{theorem} Let $G$ be a finitely generated virtually nilpotent group and $\pi$ an irreducible representation of $G.$  Then $C^*_\pi(G)$--the C*-algebra generated by $\pi(G)$-- is classifiable, has a unique trace and real rank 0.
\end{theorem}
\begin{proof} There is a trace $\tau$ on $G$ such that $C^*_\tau(G)\cong C^*_\pi(G),$ that $C^*_\tau(G)$ is simple and $\tau$ is the unique trace on $C^*_\tau(G)$ (see e.g. \cite[Section 2]{EckhardtGillaspyMcKenney2019Finite}). By the main theorem of  \cite{EckhardtGillaspyMcKenney2019Finite}, $C^*_\tau(G)$ has finite nuclear dimension.  We show it satisfies the UCT and has real rank 0.  

Let $N\trianglelefteq G$ be a normal nilpotent subgroup of finite index.
Let $\phi$ be the trivial extension of $\tau|_N$ to $G.$  By Lemma \ref{lem:trivialextsuffices} it suffices to show that $C^*_\phi(G)$ satisfies the UCT.  Since $\phi(g)=0$ for all $g\not\in N$, by Remark \ref{rem:tracetotwist} we have
\begin{equation*}
C^*_\phi(G)\cong C^*_\tau(N)\rtimes_{\alpha,\omega} G/N
\end{equation*}
By \cite[Lemmas 3.5, 3.6]{EckhardtGillaspyMcKenney2019Finite} we have 
\begin{equation*}
C^*_\tau(N) \cong \bigoplus_{i=1}^n C^*_{\tau_i}(N)
\end{equation*}
where each $\tau_i$ is a (not necessarily $G$-invariant) character on $N.$ 

The action of $G$ on the minimal central projections of $C^*_\tau(N)$ may not be transitive. But since the UCT is preserved by finite direct sums, by considering each orbit separately we may assume the action is transitive. Moreover since $C^*_\pi(G)$ is a central cutdown of $C^*_\phi(G)$ (see \cite[Proposition 4.4]{EckhardtGillaspyMcKenney2019Finite}) assuming the action is transitive does not interfere with showing real rank 0.

Let $H/N\leq G/N$ be the stabilizer of $p_1.$  By Proposition \ref{prop:transitive} we have
\begin{equation*}
C^*_\phi(G) \cong M_n\otimes (C^*_{\tau_1}(N)\rtimes_{\tilde\alpha,\tilde\omega} H/N).
\end{equation*}
So it suffices to show that 
$C^*_{\tau_1}(N)\rtimes_{\tilde\alpha,\tilde\omega} H/N$ satisfies the UCT and has real rank 0.
\newline
To avoid notational nightmares we write 
\begin{equation*}
\overline{K} = K/\ker(\tau_1) 
\end{equation*}
for any group $K$ that normalizes $\ker(\tau_1).$ Recall that $\tau_1$ descends to a trace on $\overline{N}$ and we trivially have $C^*_{\tau_1}(N)\cong C^*_{\tau_1}(\overline N).$

Let $g\in \overline N\setminus \text{FC}(\overline N).$  By \cite[Theorem 4.5]{Carey84} we have $\tau_1(g)=0.$  Hence by Remark \ref{rem:tracetotwist} we have
\begin{equation*}
C^*_{\tau_1}(\overline N) \cong C^*_{\tau_1}(\text{FC}(\overline N)) \rtimes_{\alpha',\omega'} \overline N/\text{FC}(\overline N).
\end{equation*}
for a twisted action $(\alpha',\omega').$ Since $H$ stabilizes $C^*_{\tau_1}(N)$ it follows that $\ker(\tau_1)$ is a normal subgroup of $H,$ hence $H/N\cong \overline{H}/\overline{N}$ producing the isomorphism
\begin{equation*}
C^*_{\tau_1}(\overline N)\rtimes_{\tilde\alpha,\tilde\omega} H/N \cong C^*_{\tau_1}(\overline N) \rtimes_{\tilde\alpha,\tilde\omega} \overline{H}/\overline{N}
\end{equation*}  
where we abuse notation with $(\tilde\alpha,\tilde\omega).$ Since $\text{FC}(\overline N)$ is  characteristic  we also have that $\text{FC}(\overline N)$ is normal in $\overline H.$ Hence the previous two isomorphisms combine with Remark \ref{rem:tracetotwist} to show  
\begin{equation*}
C^*_{\tau_1}(\overline N)\rtimes_{\tilde\alpha,\tilde\omega} H/N \cong C^*_{\tau_1}(\text{FC}(\overline N)) \rtimes_{\alpha'',\omega''} \overline{H}/\text{FC}(\overline{N})
\end{equation*}
for a twisted action $(\alpha'',\omega'').$ 
By Remark \ref{rem:FCcenter}, we have $C^*_{\tau_1}(\text{FC}(\overline N))$ is finite dimensional so the UCT follows from Theorem \ref{thm:vnilcase}.
Finally, real rank 0  follows from Theorem \ref{thm:vnilcase} and Corollary \ref{cor:RR0vniltwist}.
\end{proof}
\subsection{Cartan subalgebras} We construct ``explicit" Cartan subalgebras inside of C*-algebras generated by irreducible representations of finitely generated nilpotent groups.  We refer the reader to Renault's paper \cite{Renault08} for information on Cartan subalgebras and to the papers \cite{Barlak17,Li20} for their importance in the classification program.
\\\\
A group $G$ has the \textbf{unique roots property} if for any $x,y\in G$ and $n\geq 1$ we have $x^n = y^n$ implies  $x=y.$ By a classic result of Mal'cev (see \cite[Lemma 2.1]{Baumslag71Nilpotent}) every torsion free nilpotent group has the unique roots property.

\begin{lemma}\label{lem:uroots} Let $G$ be a  group with the unique roots property and $H\leq G.$   If $g\in G$ such that  the set $\{ h^{-1}gh:h\in H \}$ is finite, then $g$ is in the centralizer of $H.$  In particular this holds if $G$ is torsion free and nilpotent.
\end{lemma}
\begin{proof} Suppose the set $\{ h^{-1}gh:h\in H \}$ is finite for some $g\in G.$ Then for each $h\in H$ there is some $n\geq1$ so $h^{-n}gh^n=g$, i.e. $gh^ng^{-1}=h^n.$  By the unique roots property we have $ghg^{-1}=h.$ 
\end{proof}
The following relies on a very special fact about nilpotent groups.  Every group $G$ contains a subgroup $N$ that is maximal with respect to being both normal and abelian.  But in a \emph{nilpotent} group such a subgroup $N$ is automatically maximal with respect to being abelian.  This makes for a nice exercise, but a proof can be found in \cite[Lemma 2.22]{Clement17}. This is no longer the case for virtually nilpotent groups--indeed any finite, nonabelian simple group is a counter example--and is one reason we are not able to extend Corollary \ref{cor:buildCartan} to the virtually nilpotent case. 

\begin{proposition}\label{prop:Cartantwist} Let $G$ be a torsion free nilpotent group and $\sigma\in Z^2(G,\T).$  Then there is a subgroup $H\leq G$ such that $C^*_r(H,\sigma)$ is a Cartan subalgebra of $C^*_r(G,\sigma).$
\end{proposition}
\begin{proof} Let $\tilde{G}=\la  \lambda_g:g\in G \ra$ be the subgroup of the unitary group of  $C^*_r(G,\sigma)$ generated by the $\lambda_g.$  Since $\lambda_g\lambda_h=\sigma(g,h)\lambda_{gh}$ and $\sigma(g,h)\in Z(\tilde{G})$ it follows that $\tilde{G}$ is a central extension of a nilpotent group and therefore nilpotent. Let $\tilde H\leq \tilde{G}$ be a normal abelian subgroup that is also maximal abelian (\cite[Lemma 2.22]{Clement17}).  Set
$H\leq G$ to be the image of $\tilde{H}$ in the quotient mapping $\tilde{G}\to \tilde{G}/(\tilde{G}\cap \T)\cong G$.  We claim that $C_r^*(H,\sigma)$ is a Cartan subalgebra of $C_r^*(G,\sigma).$

The canonical conditional expectation $E:C^*_r(G,\sigma)\rightarrow C^*_r(H,\sigma)$ takes care of that requirement. Next, since  $H$ is normal in $G$, it follows that $C^*_r(H,\sigma)$ is a regular abelian subalgebra.  We show that it is maximal abelian.

Let $x\in C_r^*(G,\sigma)\cap C_r^*(H,\sigma)'.$ Decompose $x=\sum_{g\in G}x_g\lambda_g$ where $x_g\in\C$ and convergence is in the $L^2$-norm coming from the canonical trace.   Then for each $h\in H$ we have  
\begin{equation}\label{eq:cocomm}
\sum_{g\in G} x_g\sigma(g,h)\lambda_{gh}  =\sum_{g\in G} x_g\lambda_g\lambda_h  =  x\lambda_h = \lambda_hx = \sum_{g\in G}x_g\sigma(h,g)\lambda_{hg} = \sum_{g\in G}x_{h^{-1}gh}\sigma(h,h^{-1}gh)\lambda_{gh}
\end{equation} 
Since the family $\{ x_g \}_{g\in G}$ is $\ell^2$-summable, for each $g\in G$ we either have $x_g=0$ or the set $\{ h^{-1}gh:h\in H  \}$ is finite. Suppose that $x_g\neq 0.$ Then by Lemma \ref{lem:uroots} $g$ is in the centralizer of $H$ and by (\ref{eq:cocomm}) that $\sigma(g,h)=\sigma(h,g)$ for all $h\in H.$  Then 
\begin{equation*}
\lambda_g\lambda_h = \sigma(g,h)\lambda_{gh}=\sigma(h,g)\lambda_{hg}=\lambda_h\lambda_g
\end{equation*}
Hence $\lambda_g\in \tilde H$ by maximality, so $g\in H$.  Therefore $x\in C_r^*(H,\sigma).$
\end{proof}
Alex Kumjian proved Proposition \ref{prop:Cartantwist} in the case that $G$ is abelian (see \cite[Example 12]{Kumjian86}).  In fact he proved something stronger--that the Cartan subalgebras are C*-diagonals.  When $G$ is not abelian it is no longer the case that $C^*_r(H,\sigma)$ in Proposition \ref{prop:Cartantwist} is a C*-diagonal as we show in the following two examples.

\begin{example}\label{ref:nondiagonalsimple} These examples were studied in \cite{Eckhardt16} where it was shown that they are all simple A$\T$-algebras with real rank 0 and their Elliott invariants were calculated.   Let
\begin{equation*}
\text{UT}_4 = \left\{ \left[\begin{array}{cccc} 1 & x_{12} & x_{13} & x_{14}\\ 
                                                                     0 & 1 & x_{23} & x_{24}\\
                                                                     0 & 0 & 1 & x_{34}\\
                                                                     0 & 0 & 0 & 1
                               \end{array}\right] : x_{ij}\in\Z \right\}
\end{equation*}
Notice that $Z(\text{UT}_4)\cong \Z$ is the subgroup consisting of matrices with $x_{ij}=0$ when $(i,j) \neq (1,4).$ Let $G = \text{UT}_4/Z(\text{UT}_4)$ and let $c:G\to \text{UT}_4$ be the unique lifting such that the $(1,4)$ entry of $c_g$ is 0 for all $g\in G.$ Let $\gamma\in \T$ have infinite order and define $\sigma\in Z^2(G,\T)$ by
\begin{equation*}
\sigma(x,y) = \gamma^{c_xc_yc_{xy}^{-1}}
\end{equation*}
Let $H$ be the image of 
\begin{equation*}
\left\{ \left[\begin{array}{cccc} 1 & 0 & 0 & x_{14}\\ 
                                                                     0 & 1 & 0 & x_{24}\\
                                                                     0 & 0 & 1 & x_{34}\\
                                                                     0 & 0 & 0 & 1
                               \end{array}\right] : x_{ij}\in\Z \right\}
\end{equation*}
in $G$ and let $g$ be the image of $\left[\begin{array}{cccc} 1 & 0 & 0 & 0\\ 
                                                                     0 & 1 & 1 & 0\\
                                                                     0 & 0 & 1 & 0\\
                                                                     0 & 0 & 0 & 1
                               \end{array}\right]$ in $G.$
Then one checks (easily) that $C_r^*(H,\sigma)\cong C_r^*(H)\cong C(\T^2)$ is a Cartan subalgebra.  As in Section \ref{sec:trandcp} we view $C^*_r(G,\sigma)$ as the twisted crossed product $C^*_r(H,\sigma)\rtimes G/H$  giving $C^*_r(G,\sigma)$ the structure of a \emph{toplogically} principal twisted groupoid C*-algebra.  But it is not a principal groupoid. Let $\alpha:\hat H\rightarrow\hat H$ be the automorphism induced by the dual action of conjugation by $\lambda_g.$  Then $\alpha$ fixes the trivial character of $H,$ hence there is non-trivial isotropy at the trivial character.
\end{example}
That was the easiest example of a \emph{classifiable} pair that we could think of.  But the same obstruction is present in non-simple examples and easier to illustrate.
\begin{example} Let $G = \Z^2\rtimes \Z$ where the action is implemented by $\left[ \begin{array}{cc} 1 & 1\\ 0 & 1 \end{array} \right]$, i.e. $G$ is the integer Heisenberg group.  Then $C^*(\Z^2)$ is a Cartan subalgebra of $C^*(G)$ but not a C*-diagonal since the dual action has fixed points.
\end{example}
\begin{corollary}\label{cor:buildCartan} Let $G$ be a finitely generated  nilpotent group and $\pi$ an irreducible representation of $G.$  Then there is a torsion free nilpotent group $H$,  a 2-cocycle $\sigma\in Z^2(H,\T)$ and an $n\geq 1$ such that
\begin{equation*}
C^*_\pi(G) \cong M_n\otimes C_r^*(H,\sigma)
\end{equation*}
Moreover there is a subgroup $N\leq H$ such that $D_n\otimes C_r^*(N,\sigma)$ is a Cartan subalgebra of $M_n\otimes C^*(H,\sigma)$ where $D_n$ is a diagonal MASA in $M_n.$
\end{corollary}
\begin{proof}  
This is Corollary \ref{cor:RR0nilpotent} followed by Proposition \ref{prop:Cartantwist}.
\end{proof}

 \bibliographystyle{plain}

\end{document}